\documentclass[11pt]{article}

\usepackage{amsmath, amsfonts,amssymb, amsbsy}
\newtheorem{theo}{\bf Theorem}[section]
\newtheorem{cor}{\bf Corollary}[section]
\newtheorem{defi}{\bf Definition}[section]
\newtheorem{nota}{\bf Notation}[section]
\newtheorem{lem}{\bf Lemma}[section]

\begin{document}
\begin{center}
{\bf  Triple arrays from ovals in finite projective planes}

\vskip10pt
  Sunanda Bagchi\footnote{Retired professor of Indian Statistical Institute} \footnote{Corresponding author} and Bhaskar Bagchi\footnote{Retired professor of Indian Statistical Institute} \\
   Theoretical Statistics and Mathematics Unit\\
Indian Statistical Institute\\ Bangalore 560059, India\\
e-mails : sunandab355@gmail.com,  bhaskarbagchi53@gmail.com
\end{center}


This paper has been posted in arxiv: 2607.20275[math.CO].
\begin{abstract}
In this paper we prove that whenever a projective plane of odd order $n$ contains an oval, it may be used to construct a triple array with $n+1$ rows, $n^2$ columns and $n(n+1)$ symbols. In particular, for any odd prime power $s$, we use the projective plane over the field of order $s$ to construct an explicitly given  $(s+1) \times s^2$ triple array on $s(s+1)$ symbols. This is only the third infinite series of non-trivial triple arrays (up to transposition) to be found. Our construction proves a recent conjecture of Gordeev and \"Ohman (2026) in the case of odd prime powers.
\end{abstract}

\section{Introduction}

For integers $p \geq 2, \;  q \geq 2, \; v \geq 1$, Preece et al \cite{pree} defined a $p \times q$ \textbf{triple array} $A$ on $v$ symbols  as an allocation of   $v$ symbols to the $pq$ positions of a $p \times q$ rectangular grid such that (i) $A$ is \textbf{binary} (i.e., the symbols in each row and each column are distinct), (ii) $A$ is \textbf{equi-replicate} (i.e. there is a constant $r$, the \textbf{replication number} of the array, such that each symbol occurs in exactly $r$ positions of the array; clearly, we must have $r= (pq)/v$) and, further, there are three constants $\lambda_{RR}, \lambda_{CC} \; \rm{and} \; \lambda_{RC}$ such that the array $A$ satisfies the following three conditions:
\begin{itemize}
\item (RR) any two distinct rows of $A$ have exactly $\lambda_{RR}$ common symbols,

\item (CC) any two distinct columns of $A$ have exactly $\lambda_{CC}$ common symbols,

\item (RC) any row and any column of $A$ have exactly $\lambda_{RC}$ common symbols.
\end{itemize}
These three constants may be called the \textbf{intersection numbers} of the triple array.

Notice that the transpose of a $p \times q$ triple array is a $q \times p$ triple array.  Preece et al \cite{pree} constructed $s \times (s+1)$ triple arrays on $2s$ symbols (dubbed the Paley triple arrays) for all odd prime powers $s > 3$. This generalizes previous constructions due to Agrawal \cite{agar} (the case $s \equiv 3 \pmod 4$) and Bagchi \cite{su2} (the case $s \equiv 5 \pmod 8$). Nilson and Cameron \cite{nil} used Menon difference sets to construct $m(2m-1) \times m(2m+1)$ triple arrays on $(2m-1)(2m+1)$ symbols for all integers $m>1$ such that the square free part of $m$ divides $6$.

Since by definition, any triple array is binary, and since $v$ divides $pq$ in any $p \times q$ triple array on $v$ symbols, we must have $\max{(p,q)} \leq v \leq pq$. We say that the triple array is \textbf{non-trivial} if $\max{(p,q)} < v < pq$. (Remember that the conditions $p>1, q > 1$ are built into our definition of triple arrays.) Prior to this paper, the series  mentioned in the previous paragraph were  the only known infinite series of non-trivial triple arrays, up to transposition. In this paper we use the classical finite projective planes $PG(2,s)$ to construct $(s+1) \times s^2$ triple arrays on $s(s+1)$ symbols for all odd prime powers $s$. The case where $s$ is a power of $2$ remains open.

If $A$ is a binary array then the \textbf{row design} $D_R^A$ and the \textbf{column design} $D_C^A$ of $A$ are the block designs defined as follows. Both designs have the $v$ symbols of $A$ as the points. The row design has, corresponding to each of the $p$ rows of $A$, a block consisting of the un-ordered set of $q$ symbols occurring in this row. The \textbf{column design} of $A$ is defined analogously using the columns. (Some authors use the duals of these designs as their definition of the row- and column- designs.) These two designs together are called the \textbf{marginals} of the binary array.

Recall that a $2-(v,k,\lambda)$ \textbf{design} is an incidence system on $v$ points with $k$ points per block such that any two distinct points are together in $\lambda$ blocks. A standard counting argument shows that any such design has ancillary parameters $b,r$ such that $b$ is the total number of blocks and $r$ is the number of blocks through each point. These ancillary parameters of a $2$-design are given in terms of the main parameters $v,k,\lambda$ by the equations $bk=rv$ and $r(k-1)=\lambda(v-1)$. The number $r-\lambda$ is called the \textbf{order} of the design. The \textbf{complementary design} $\overline{D}$ of a  $2-(v,k,\lambda)$ design $D$ is the $2-(v,v-k, b-2r + \lambda)$ design whose blocks are the complements (in the full point set of $D$) of the blocks of $D$.\, $D$ and $\overline{D}$ have the same order. All results on $2$-designs used here may be found in the beautiful monograph \cite{cam} by Cameron and van Lint.

 As an immediate consequence of the definition of triple arrays, we have the following well known result.
\begin{theo} \label{theo 1.1}  For any triple array A , the dual of $D_R^A$ is a $2-(p,(pq)/v, \lambda_{RR})$ design and the dual of $D_C^A$ is a $2-(q, (pq)/v, \lambda_{CC})$ design.
\end{theo}

  Notice that the replication number of the dual of $D_R^A$ (respectively, of $D_C^A$) is $q$ (respectively $p$). Therefore, invoking the relations between the parameters of $2$-designs given in the previous paragraph, one concludes the following result (see, for example, the discussion in Gordeev and \"Ohman \cite{gor2}, section 2.3.)

    \begin{cor}The intersection numbers of any $p  \times q$ triple array on $v$ symbols are given by:

   $$\lambda_{RR} =\frac{q(pq-v)}{(p-1)v}, \;\; \lambda_{CC}= \frac{p(pq-v)}{(q-1)v}, \;\; \lambda_{RC} = \frac{pq}{v}.$$
   \end{cor}

   Since the intersection numbers of a triple array are integers, it follows that we have:
    \begin{cor} The parameters $v,p,q$ of a triple array must satisfy the following divisibility conditions: (i) $(p-1)v$ divides $q(pq-v)$, (ii) $(q-1)v$ divides $p(pq-v)$ and (iii) $v$ divides $pq$.
   \end{cor}
   Recall that a \textbf{linked block design} (LBD) is defined to be a block design whose dual (obtained by interchanging points and blocks) is a  $2$-design. Thus the marginals of any triple array form a pair of LBDs on a common set of $v$ points, with $p$ and $q$ blocks of sizes $q$ and $p$ respectively, satisfying the following condition: any block of the  first  LBD meets  any  block  of  the  second  in  $\lambda_{RC} \, \rm{points.} $

     Essentially following Definition 7 in \cite{gor2}, we introduce:
      \begin{defi} A $p \times q$ \textbf{un-ordered triple array} on $v$ symbols is a pair $(D_R,D_C)$ of LBD's satisfying the following conditions: (i) $D_R$ and $D_C$ have a  common point-set of size  $v$, (ii) $D_R$ has  $p$ blocks of size $q$ each and $D_C$ has  $q$ blocks of size $p$ each, and (iii) there is a constant $\lambda_{RC}$ such that each block of $D_R$ meets each block of $D_C$ in a set of size $\lambda_{RC}$. (Using (i) and (ii), it is easy to see that the average intersection size between blocks of $D_R$ and $D_C$ is $(pq)/v$. Therefore, (iii) implies that we must have $\lambda_{RC}= (pq)/v$.)
      \end{defi}

      Thus, by Theorem 1.1, the pair $(D_R^A,D_C^A)$ of marginals of any $p \times q$ triple array $A$ on $v$ symbols is a  $p \times q$ un-ordered triple array on $v$ symbols. We have the following trivial converse to this observation.
      \begin{theo} The pair of marginals of a rectangular array $A$ is an un-ordered triple array if and only if $A$ is a triple array. \end{theo}

      Indeed, if $A$ is a $p \times q$ array on $v$ symbols such that the marginals of $A$ are LBDs, then clearly $A$ is binary and satisfies (RR) and (CC). Since these LBDs are equi-replicate (necessarily with the common replication number $r=(pq)/v$) and $A$ is binary, it follows that $A$ is equi-replicate with replication number $r$.  Finally, the condition (iii) on un-ordered triple arrays ensures that  any such array $A$ satisfies condition (RC) as well. It is easy to see that the divisibility conditions of Corollary 1.2 already hold for the parameters $p,q,v$ of an un-ordered triple array.

       Following \cite{gor2} again, we introduce:
       \begin{defi} An un-ordered triple array $(D_R, D_C)$ is said to be orderable if there is a triple array $A$ with marginals $D_R^A=D_R$, $D_C^A=D_C$. \end{defi}

       With these definitions, Gordeev and \"Ohman \cite{gor2} split the construction problem of triple arrays into two sub-problems: (i) construct un-ordered triple arrays, and (ii) given an un-ordered triple array, determine if it is orderable.

    The first author used matrix methods to prove \cite{su2} that the parameters of any non-trivial triple array must satisfy the  inequality $v \geq p + q - 1$. This was independently reproved in  McSorley et al \cite{mcsor}. Recently Gordeev et al \cite{gor1} used a counting argument to give a new proof of this inequality. Any of these proofs goes through verbatim to show that any $p \times q$ unordered triple array $(D_R, D_C)$ on $v$ symbols satisfies the same inequality. A triple array satisfying the equality $v=p+q-1$ is said to  be \textbf{an extremal triple array}. Likewise we may  define an \textbf{ extremal un-ordered triple array} to be one satisfying this equality. Bagchi and Shah \cite{su1} proved that any  triple array, when used as a statistical design for the elimination of two-way heterogeneity, satisfies very strong optimality properties. (Actually, \cite{su1} states this result only for extremal triple arrays, but, with minor modifications, the same proof goes through for any triple array.) Thus the construction of  triple arrays is important for their applications to design of statistical experiments. Notice that the three known series of triple arrays (including the series constructed here) are extremal. Gordeev and \"Ohman \cite{gor2} used a computer aided search to find finitely many non-extremal triple arrays.

  Recall that a $2-(v,k, \lambda)$ design $D$ is said to be \textbf{symmetric} if the dual design $D^*$ (obtained from $D$ by interchanging points and blocks) is also a $2-$design (necessarily with the same parameters $v,k, \lambda$). A $2-(v,k, \lambda)$ design is symmetric if and only if its parameters satisfy the relation $\lambda (v-1)= k(k-1)$. The complement $\overline{D}$ of a symmetric design is again symmetric.

   The following construction of un-ordered triple arrays is due to Agrawal \cite{agar2}.

  \begin{nota} Let $D$ be a symmetric $2-(v, k,\lambda)$ design of order $\geq 2$ and let $x$ be a point of $D$. We denote by $\mathbb{A}(D,x)$ the pair $(D_R, D_C)$ of block designs where (i) the blocks of $D_R$ are the complements (in the point set of $D$) of the blocks of $D$ through $x$, and (ii) the blocks of $D_C$ are the blocks of $D$ not passing through $x$. Since the dual of $D$ is a $2-$design with the same parameters, it is easy to verify that $\mathbb{A}(D,x)$ is an extremal $k \times (v-k)$ un-ordered triple array on $v-1$ symbols. Its intersection numbers are $\lambda_{RR} = v-2k+\lambda,\, \lambda_{CC}= \lambda, \,\lambda_{RC}=k-\lambda$. \end{nota}

   McSorley et al \cite{mcsor} showed that with any $p \times q$ extremal triple array, we can attach a symmetric $2-(p+q, p, \lambda)$ design (where, necessarily, $\lambda= p(p-1)/(p+q-1)$). Their argument goes through verbatim to prove the direct part of the following stronger result. Moreover, running this argument backwards, one obtains a proof of its converse part.

   \begin{theo} \label{theo 1.3} If $(D_R, D_C)$ is a $p\times q$ extremal un-ordered triple array then there is a symmetric $2-(p+q, p, \lambda)$ design $D$ of order $\geq 2$ and a point $x$ of $D$ such that $(D_R,D_C)= \mathbb{A}(D,x)$. Conversely, if $D$ is a symmetric $2$-design of order $\geq 2$ and $x$ is a point of $D$ then $ \mathbb{A}(D,x)$  is an extremal un-ordered triple array.
   \end{theo}

   In terms of the language and notation introduced above, a famous conjecture of Agrawal \cite{agar2} may be stated as follows (this is stated as Conjecture 18 in Gordeev and \"Ohman \cite{gor2}):

   \textbf{Agrawal's conjecture}: Except when $(p,q)=(3,4)\, \rm{or} \, (4,3)$, any $p \times q$ extremal un-ordered triple array is orderable. That is, if the symmetric design $D$ is  of order $>2$ then, for any point $x$ of $D$, there is a (extremal) triple array $A$ such that the pair of marginals of $A$ is $\mathbb{A}(D,x)$.

    Notice that if the pair of marginals of an array $A$ is $\mathbb{A}(D,x)$, then $\mathbb{A}(\overline{D},x)$ is the pair of marginals of the transposed array $A^T$. Thus, $\mathbb{A}(D,x)$ is orderable if and only if $\mathbb{A}(\overline{D},x)$ is orderable.

    Up to isomorphism, the only symmetric $2$-designs of order $2$ are the Fano plane (i.e., the unique $2-(7,3,1)$ design) and its complement (the $2-(7,4,2)$ design). Using this well known fact in tandem with Theorem 1.3, it is not hard to confirm that $(p,q)=(3,4), (4,3)$ are genuine exceptions to this conjecture : there is no extremal triple array with these parameters.

   Agrawal's conjecture is  wide open. In this paper we address the case $\lambda =1$ of Agrawal's conjecture. Recall that the\textbf{ finite projective planes} are just the symmetric $2$-designs with $\lambda=1$. Thus a projective plane of order $n$ is nothing but a $2-(n^2+n+1, n+1, 1)$ design. The blocks of a projective plane are usually called its lines.

    As a special case of Theorem 1.3, the marginals of any $(n+1) \times n^2$ extremal triple array must be of the form $\mathbb{A}(\Pi, \infty)$ for some finite projective plane $\Pi$ of order $n$ and some point $\infty$ of $\Pi$.  By Agrawal's conjecture, for $n>2$,  extremal triple arrays with these parameters ought to be co-existent with projective planes of order $n$. (As already remarked, they do not exist for $n=2$.) According to a long-standing and widely believed folklore conjecture, projective planes of order $n$ should not exist unless $n$ is a prime power. Thus, it is safe to conjecture that such arrays exist only for prime powers $n$.

    An \textbf{oval} $C$ in a projective plane $\Pi$ of order $n$ is a set of $n+1$ points no three of which are collinear. That is, any line $\ell$ of $\Pi$ meets $C$ in $0,1$ or $2$ points. The line $\ell$ is said to be a passant, tangent or a secant line to $C$, accordingly. Since each point in a plane of order $n$ is on $n+1$ lines, and each point on an oval is on exactly $n$ secant lines (namely the distinct lines joining it to the remaining $n$ points on the oval) it follows that each point of an oval $C$ is on a unique tangent line, so that there is a total of $n+1$ tangent lines to $C$. When the order $n$ of $\Pi$ is odd, no three of the tangent lines are concurrent, i.e., the set $C^*$ consisting of the tangent lines to $C$ is an oval in the dual projective plane $\Pi^*$. (These results on ovals may be found in Chapter XII of \cite{hu}.) Using this geometry, we prove the main result (Theorem 2.1 below) of this paper: if $\Pi$ is a projective plane of \textbf{odd order} $n$, and $\infty$ is a point of $\Pi$ such that $\infty \in C$ for some oval $C$ of $\Pi$, then the un-ordered triple array $\mathbb{A}(\Pi, \infty)$ is orderable. Hence there is an $(n+1) \times n^2$ extremal triple array whenever $n$ is an odd number such that there is a projective plane of order $n$ containing at least one oval.\footnote{We do not know if there are finite projective planes of odd order  which contain no ovals! However, in \cite{pent}, Penttila et al found that some of the known projective planes of the even order 16 contain no ovals.}  It may be noted that the orders of all the known finite projective planes are prime powers.

   For prime powers $s$, let $PG(2,s)$ denote the projective plane (of order $s$) over the field of order $s$. When $s$ is odd, a famous theorem of Segre \cite{seg} states that the only ovals in $PG(2,s)$ are the \textbf{conics} (i.e., the zero sets of non-degenerate ternary quadratic forms, viewed projectively). There are $s^2(s^3-1)$ such ovals, and each point of $PG(2,s)$ is on $s^2(s^2-1)$ of them. (When $s>4$ is a power of two, there are even more ovals in $PG(2,s)$.) Therefore the main result of this paper implies that whenever $s$ is an odd prime power, $\mathbb{A}(PG(2,s), \infty)$ is orderable for any point $\infty$ of $PG(2,s)$, and hence there is an $(s+1) \times s^2$ extremal triple array. (Since the automorphism group of $PG(2,s)$ is transitive on points, the choice of the base point $\infty$ is irrelevant here.) This proves  Conjecture 54 of Gordeev and \"Ohman \cite{gor2} in the affirmative in case of odd prime powers.

   The geometric behaviour of ovals in projective planes of even order is drastically different. In this case, the tangents to an oval pass through a common point (the nucleus of the oval). As a result, our construction fails dramatically for even orders. For a while we suspected that $(n+1) \times n^2$ extremal triple arrays might not exist for any even number $n$. However, in \cite{gor1} and \cite{gor2}, the authors report finding  a large number of non-isomorphic arrays with these parameters for $n=4$ and $n=8$ via a computer aided search.

\section{Construction}
 We recall (see Bailey et al \cite{bai}) that a \textbf{double array} (respectively \textbf{sesqui-array}) is a binary equi-replicate array satisfying conditions (RR) and (CC) (respectively (RR) and (RC)) in the definition of a triple array, but not necessarily the remaining condition in this definition. We begin by observing that it is relatively easy to construct double arrays and sesqui-arrays with our parameters, but these methods never yield triple arrays. Thus, the construction of the triple arrays requires deeper methods.

Recall that a $(v,k, \lambda)$ \textbf{difference set} $D$ is a subset of size $k$ in a finite group $G$ of order $v$ (written additively, even though in general $G$ may not be abelian) such that every non-zero element of $G$ can be written as the difference of two elements of $D$ in exactly $\lambda$ ways. The \textbf{development} of $D$ (denoted by $\rm{dev}(D)$) is the block design (with point set $G$) whose blocks are the $v$ translates $g+D, g \in G$, of $D$. Thus $D$ is a $(v,k,\lambda)$ difference set if and only if $\rm{dev}(D)$ is a symmetric $2-(v,k, \lambda)$ design (admitting $G$ as a regular group of automorphisms). The difference set is called abelian (respectively cyclic) if the ambient group is abelian (respectively cyclic). We say that the difference set admits $-1$ as a multiplier if the set $-D := \{-x : x \in D \}$ is a translate of $D$.

Given a $(v,k, \lambda)$ difference set $D$ in a finite group $G$, define the $k \times (v-k)$  array $A^D$ (with rows and columns indexed by $D$ and $G\setminus D$ respectively) such that its entries are $A^D(x,y) = x-y, \, x \in D, y \in G\setminus D$. In Nilson and Cameron \cite{nil}, the authors show that (a) $A^D$ is always a $k \times (v-k)$ double array on $v-1$ symbols, and (b) when $D$ is abelian, $A^D$ is a (extremal) triple array if and only if $D$ admits $-1$ as a multiplier.

It is well known (see the references in \cite{nil}) that, whenever the square free part of an integer $m$ divides $6$, there is a $(4m^2, 2m^2-m, m^2-m)$ abelian difference set admitting $-1$ as a multiplier. The above construction applied to these difference sets yields the Nilson-Cameron series of extremal triple arrays. McFarland \cite{mcfar} constructed (as a special case of a remarkable general construction of difference sets) a $(4000, 775, 150)$ abelian difference set with multiplier $-1$. This yields a sporadic example of a $775 \times 3225$ extremal triple array.

A famous construction of  Singer \cite{sing} yields a cyclic difference set $D_s$ such that $\rm{dev}(D_s)=PG(2,s)$ for any prime power $s$. Applying the Nilson-Cameron construction to these difference sets yield $(s+1) \times s^2$ double arrays on $s^2+s$ symbols for all prime powers $s$. However, cyclic difference sets never admit $-1$ as a multiplier (see \cite{bau} for example). Therefore this construction never yields triple arrays.

\textbf{The construction of a  sesqui- array}

We begin by observing that, for any positive integer $n$, it is fairly easy to construct $(n+1) \times n^2$ sesqui-arrays on $n(n+1)$ symbols and intersection numbers $\lambda_{RR}=n(n-1), \; \lambda_{RC}=n$. To this end, let's put $I = \{ 0,1, \cdots, n \}$ and $J= \{ 1, 2, \cdots , n \}$. Consider the $(n+1) \times n^2$ array $B$ (with entries from $I \times J$ and rows and columns indexed by $I$ and $J \times J$ respectively) whose entries are $B(i, (j,k)) := (i+j, k)$ \; ($i \in I, (j,k) \in J \times J$). Here, the addition in the first co-ordinate is modulo $n+1$. Note that the blocks of the row design of $B$ are the sets $(I\setminus \{ i \}) \times J, \; i \in I.$ The blocks of the column design of $B$ are the sets $I \times \{ j \}, \; j \in J,$ each of these sets repeated $n$ times. Therefore, $B$ is indeed a sesqui-array with parameters as stated. However, $B$ is never a triple array. Any two distinct columns of $B$ have $0$ or $n+1$ common symbols, and both intersection numbers occur. Thus, to create triple arrays with these parameters, we need deeper ideas. We succeed only when $n$ is an odd prime power.

\textbf{Construction of a triple array.}

 We begin with:

\begin{nota} In the following, $\Pi$ is a finite projective plane, say of odd order $n$. Let $\infty$ be a point of $\Pi$ and let $C$ be an oval of $\Pi$ such that $\infty \in C$. Let $\ell_\infty$ be the tangent line to $C$ at $\infty$. Also, let $\ell_i, \, i \in \ell_\infty$, be the $n+1$ tangent lines to $C$. We choose this indexing in such a way that $i \notin \ell_i$ for all $i \neq \infty$. This may be achieved by choosing a derangement (i.e., a fixed point free permutation) $\sigma$ of $\ell_\infty \setminus \{ \infty \}$ and letting $\ell_i$ be the tangent line $\neq \ell_\infty$ of $C$ passing through the point $\sigma(i)$ for all $i \neq \infty$.

 For $i,j \in \ell_\infty$, we define the point $x_{ij}$ as follows. (i) $x_{ii}$ is the point of $C$ where $\ell_i$ is tangent, and (ii) for $i \neq j$, $x_{ij}=x_{ji}$ is the common point of $\ell_i$ and $\ell_j$. Thus, $x_{\infty, \infty}=\infty$ and (since each point of $C$ is on a unique tangent), we have $x_{ij} \notin C$ for $ i \neq j$ and

  \begin{align} C=\{ x_{ii}: i \in \ell_\infty \}. \label{eq1} \end{align}

 Also, since each point outside $C$ is in $0$ or $2$ tangents, it follows that for each $i \in \ell_\infty$,

    \begin{align} \ell_i= \{x_{ij}: j \in \ell_\infty \}.\label{eq2} \end{align}

  Let's define the line $m_{ij}$ through $\infty$ as follows. (i) $m_{i, \infty}=m_{\infty, i}=\ell_\infty$ and (ii) for $i \neq \infty, j \neq \infty$, $m_{ij}$ is the line joining the points $\infty$ and $x_{ij}$. (In the latter case, we have $x_{ij} \neq \infty$, so that this is well defined.)
\end{nota}

In terms of this notation, we have
\begin{lem} Let $\Pi$ be a projective plane of odd order $n$. Then,

(a) $m_{ii}, i \in \ell_\infty,$ are all the lines of $\Pi$ through the point $\infty$, and

(b) for each fixed $i \in \ell_\infty \setminus \{ \infty \}$,\, $m_{ij}, j \in \ell_\infty,$ are all the lines of $\Pi$ through the point $\infty$.

\end{lem}

\textbf{Remark}. When the order $n$ of $\Pi$ is even, the tangents to an oval $C$ are concurrent at a point. It follows that, for all $i \neq j \in \ell_\infty$, we have $m_{ij} = \ell_\infty$. Thus, Part (b) of Lemma 2.1 fails spectacularly for planes of even order.
We might begin Notation 2.1 with a dual oval $\{\ell_i : i \in \ell_\infty \}$ (i.e., a set of $n+1$ lines of $\Pi$ no three of which are concurrent) and take $C$ to be the set of all points $x$ of $\Pi$ such that $x$ belongs to exactly one line $\ell_i$. (Thus, when $n$ is odd, this coincides with the version of Notation 2.1 given above. However, when $n$ is even, $C$ is a line.) With this modification, Part (b) of Lemma 2.1 remains valid for $n$ even, but we get $m_{ii}= C$ for all $i \neq \infty$, so that Part (a) fails in this case. However, both parts of this Lemma are crucial for our construction. This explains why our construction fails for even $n$.

\textbf{Proof of Lemma 2.1}:
Fix an arbitrary line $m$ through $\infty$. It suffices to show that (a) there is an $i \in \ell_\infty$ such that $m=m_{ii}$, and, (b) for any fixed $i \neq \infty$, there is a $j$ such that $m=m_{ij}$. Both assertions are trivial when $m= \ell_\infty$. So take $m \neq \ell_\infty $. To prove (a), note that, since $\infty \in m \cap C$ and $\ell_\infty$ is the only tangent line to $C$ through $\infty$, it follows that $m$ is a secant to $C$. Thus, in view of Equation \ref{eq1}, $m \cap C = \{ \infty,x_{jj} \}$ for some index $j \neq \infty$. Then $m=m_{jj}$ for this index $j$. This proves part (a).

To prove (b), note that, since $\ell_i$ does not pass through $\infty$,  $\ell_i \neq m$. In view of Equation \ref{eq2}, the common point of the lines $\ell_i$ and $m$ equals $x_{ij}$ for some index $j$. Then $m = m_{ij}$. This proves Part (b). $\Box$

We shall also need:
\begin{nota}

For $i \in \ell_\infty$, let $r_i$ denote the complement in $\Pi$ of the line $m_{ii}$. For  $i \in \ell_\infty, \, t \in \ell_\infty \setminus \{ \infty \}$, we introduce the sets $r_i(t)$ as follows.

\begin{equation*}
r_i(t):= \begin{cases} r_i^0(t) & \text{if} \; i=\infty \\
                      \{ t \} \cup r_i^0(t) & \text{otherwise,}
\end{cases}
\end{equation*}

where the sets $r_i^0(t)$ are defined by:

\begin{equation*}
r_i^0(t) :=
\begin{cases} m_{tt} \setminus \{\infty\}& \text{if}\; i = \infty, \\
\ell_i \setminus \{ x_{ii}, x_{i\infty} \} & \text{if}\; i=t, \\
m_{it} \setminus \{\infty, x_{it} \}  &  \text{if} \; i \neq \infty,t.
\end{cases}
\end{equation*}

\end{nota}

\textbf{Remark}: The proof of our main result (Theorem 2.1 below) begins by constructing  $(n+1) \times n$ sub-arrays $A_t, t \in \ell_\infty \setminus \{ \infty \}$. The triple array $A$ will be obtained by juxtaposing these sub-arrays. The sets $r_i(t), i \in \ell_\infty$, of Notation 2.2 and the sets $c_j(t), j \in \ell_\infty \setminus \{ \infty \}$, of Notation 2.3 below  will be the blocks of the two marginals of $A_t$. This is the importance of these sets. The properties of these sets established in the following two lemmas ensure that (i) there indeed is a sub-array with these marginals (indeed, these lemmas imply that the array $A_t$ is uniquely determined by the requirement that it has these prescribed marginals) and (ii) the final array $A$ thus created has the desired marginals predicted by Theorem 1.3.

\begin{lem}

(a) For each fixed $i \in \ell_\infty$, $r_i(t)$, for $t \in \ell_\infty \setminus \{ \infty \}$, are $n$ pairwise disjoint sets  of size $n$ each, and \, $\bigcup_{t \in \ell_\infty \setminus \{ \infty \}} r_i(t)= r_i$. \\
(b) For each fixed $t \in \ell_\infty \setminus \{ \infty \}$, the $n+1$ sets $r_i^0(t)$, for $ i \in \ell_\infty,$ are pairwise disjoint ($n$ of them of size $n-1$ each, and one is of size $n$) and $\bigcup_{i \in \ell_\infty}r_i^0(t)=r_\infty.$
\end{lem}

\textbf{Proof:} Since each line of $\Pi$ has size $n+1$ and any two distinct lines meet in a unique point, parts (a) and (b) are immediate from Lemma 2.1. Perhaps the only statements that require comment are (i) the assertion in Part (a) that, for fixed $i$, the union of the sets $r_i(t), t \in \ell_\infty \setminus \{ \infty \},$ is the set $r_i$, and (ii) the corresponding statement in Part (b). Indeed, once we verify that the $n$ sets in Part (a) are  pairwise disjoint sets of size $n$ each, it follows that their union is of size $n^2$. It is also easy to see that each of these sets is disjoint from the line $m_{ii}$. Hence their union is contained in the complement $r_i$ of $m_{ii}$. Since $r_i$ is also of size $n^2$, the equality follows. The assertion in Part (b), for fixed $t$,  on the union of the $n+1$ sets $r_i^0(t), i \in \ell_\infty$, may be proved similarly.  $\Box$

\begin{nota}
For any $t \in \ell_\infty \setminus \{ \infty \}$, we arbitrarily index the lines except $\ell_\infty$ through $t$ as $c_j(t), j \in \ell_\infty \setminus \{ \infty \}$. We define a function $\sigma_t : \ell_\infty \rightarrow \ell_\infty$ as follows. We define $\sigma_t(\infty)= \infty$. Next let $i \in \ell_\infty \setminus \{ \infty \} $. Then $x_{it} \notin \ell_\infty$. Hence, there is a unique index $j \in \ell_\infty \setminus \{ \infty \}$ such that $x_{it} \in c_j(t)$. Namely, this is the index for which $c_j(t)$ is the line joining the points $t \in \ell_\infty \setminus \{ \infty \}$ and $x_{it}$. We put $\sigma_t(i)=j$. In other words, for $i \neq \infty,\, \sigma_t(i)$ is the index $j$ determined by the requirement that the three lines $\ell_t$, $m_{it}$ and $  c_j(t)$ are concurrent.
\end{nota}

\begin{lem}
 The function $\sigma_t$  is a permutation of $\ell_\infty$. It satisfies:  ($\sigma_t(\infty)=\infty$, and) for all $i \in \ell_\infty, \;  j \in \ell_\infty \setminus \{ \infty \}$, we have
\begin{equation*}
 | r_i^0(t) \cap c_j(t)| = \begin{cases} 0 & \text{if} \; j = \sigma_t(i),\\ 1 & \text{otherwise.}
\end{cases} \end{equation*}

\end{lem}

\textbf{Proof}

  Fix $j \in \ell_\infty \setminus \{\infty \}$. Let $y_{jt}$ be the unique point in $c_j(t) \cap \ell_t$. Since $t \notin \ell_t$ by our choice of the indexing of the tangents to $C$, and  $c_j(t) \cap \ell_\infty = \{ t \}$, it follows that $y_{jt} \notin \ell_\infty$. Therefore, by Lemma 2.1(b) there is a unique index $i \neq \infty$ such that $m_{it}$ is the line joining the points $\infty$ and $y_{jt}$. By construction, the lines $\ell_t,\, m_{it} \, \text{and} \, c_j(t)$ are concurrent at the point $y_{jt}$. Hence $\sigma_t(i)=j$ for this choice of $i$. Since $j \in \ell_\infty \setminus \{ \infty \}$ was arbitrary and since, also, $\sigma_t(\infty)=\infty$, it follows that the map $\sigma_t : \ell_\infty \rightarrow \ell_\infty$ is onto. Therefore, by the pigeonhole principle, $\sigma_t$ is a bijection.

Since $c_j(t) \cap \ell_\infty = \{ t \}$ and $m_{tt} \cap \ell_\infty = \{ \infty \}$, it follows that $c_j(t) \neq m_{tt}$. Hence the lines $c_j(t)$ and $m_{tt}$ meet at a unique point, which is not $\infty$. Since $r_\infty^0(t) = m_{tt} \setminus \{ \infty \}$, it follows that $|r_\infty^0(t) \cap c_j(t)|=1.$

Next let $i,j, t \in \ell_\infty \setminus \{ \infty \}$. Recall that $\sigma_t(i)=j$ if and only if $x_{it} \in c_j(t)$. Since the line $c_j(t)$ meets each of the lines $\ell_t$ and $m_{it}$ at  a unique point, and since neither of the points $ x_{t\infty}, \infty \in \ell_{\infty}$ belongs to $c_j(t)$, the result is now immediate because $r_i^0(t) = \ell_t \setminus \{ x_{it}, x_{t\infty} \}$ when $i=t$ and $r_i^0(t) = m_{it} \setminus \{x_{it}, \infty \} $ otherwise.   $\Box$

We shall now construct the required triple array. Recall from Notation 1.1 that $\mathbb{A}(\Pi, \infty) = (D_R, D_C)$ is the unordered triple array where the blocks of $D_R$ are the complements (with respect to the full point set of $\Pi$) of the lines of $\Pi$ through $\infty$, and the blocks of $D_C$ are the lines of $\Pi$ not passing through $\infty$.

\begin{theo} Let $\Pi$ be a projective plane of odd order $n$. Let $\infty$ be a point of $\Pi$, and suppose there is an oval $C$ of $\Pi$ such that $\infty \in C$. Then the un-ordered extremal triple array $\mathbb{A}(\Pi, \infty)$ is orderable. That is, there is an $(n+1) \times n^2$ extremal triple array $A$ such that $\mathbb{A}(\Pi, \infty)$ is the pair of marginals of $A$.
\end{theo}

\textbf{Proof:} Fix $t \in \ell_\infty \setminus \{ \infty \}$. We construct the $(n+1) \times n$ array $A_t$ as follows. The rows and columns of $A_t$ are indexed by the points in $\ell_\infty$ and $\ell_\infty \setminus \{ \infty \}$, respectively. For $i \in \ell_\infty, j \in \ell_\infty \setminus \{ \infty \}$, the $(i,j)-$th entry of $A_t$ is given by

\begin{equation} \label{equation} A_t(i,j):= \begin{cases} t & \text{if} \,  j = \sigma_t(i), \\
z_{i,j}(t) & \text{otherwise;}
\end{cases}
\end{equation}
where $z_{i,j}(t)$ is the unique element of $r_i^0(t) \cap c_j(t)$ in the second case. This is well defined by Lemma 2.3. Indeed, by the proof of this lemma, $z_{i,j}(t)$ is the common point of the lines  $c_j(t)$ and $\ell$, where (i) $\ell=m_{tt}$ when $i=\infty$, (ii)  $\ell = \ell_t$ when $i=t$, and (iii)   $\ell=m_{it}$ when $i \neq \infty,t$.

Notice that, since, by Lemma 2.2(b), each $r_i^0(t)$ is disjoint from $\ell_\infty$,  the only element of $\ell_\infty$ which occurs as an entry of the array $A_t$ is the point $t$. Also, as $\sigma_t$ is a permutation of $\ell_\infty$ and $\sigma_t(\infty)=\infty$ (Lemma 2.3), it follows that the entry $t$ occurs in exactly one position of each column of $A_t$. Also, $t$ occurs exactly once in each row of $A_t$, except for the $\infty-$th row where $t$ does not occur. Since the sets $r_i^0(t), i \in \ell_\infty$, are pairwise disjoint (Lemma 2.2(b)) and $t \notin r_i^0(t)$ for all $i$ (indeed, each $r_i^0(t)$ is disjoint from $\ell_\infty$), it follows that the $n+1$ entries in each column of $A_t$ are distinct. Also, as the lines $c_j(t), j \in \ell_\infty \setminus \{ \infty \}$, meet pairwise at $\{ t \}$, the $n$ entries in each row of $A_t$ are distinct. Thus $A_t$ is an $(n+1) \times n$ binary array. Letting $D_R(t)$ and $D_C(t)$ denote the row-design and the column-design (respectively) of $A_t$, it follows that each block  of $D_R(t)$ is of size $n$ and each block of $D_C(t)$ is of size $n+1$. But, by construction, the $i-$th block of $D_R(t)$ is contained in the set $r_i(t)$ of size $n$ (Lemma 2.2(a)) and the $j-$th block of $D_C(t)$ is contained in the set $c_j(t)$ of size $n+1$. Therefore the blocks of $D_R(t)$ are the sets $r_i(t), i \in \ell_\infty$, and the blocks of $D_C(t)$ are the sets $c_j(t), j \in \ell_\infty \setminus \{ \infty \}$.

Finally, define the $(n+1) \times n^2$ array $A$ by $A :=(A_t : t \in \ell_\infty \setminus \{ \infty \})$ obtained by juxtaposing the $n$ sub-arrays $A_t$ in a row. Thus, the rows of $A$ are indexed by the points in $\ell_\infty$, the columns of $A$ are indexed by the ordered pairs in $(\ell_\infty \setminus \{ \infty \}) \times ( \ell_\infty \setminus \{ \infty \})$. For $i \in \ell_\infty, j,t \in \ell_\infty \setminus \{ \infty \}$, the entry of the array $A$ in the intersection of the $i$-th row and $(t,j)$-th column is given by $A(i, (t,j)):= A_t(i,j)$.

Since, by Lemma 2.2(a), $r_i$ is the disjoint union of the $i$-th blocks of   $D_R(t), t \in \ell_\infty \setminus \{ \infty \}$, it follows that $A$ is a binary array and  the $i-$th block of the row design $D_R^A$ of $A$ is $r_i$. Since $r_i$ is by definition the complement of the line $m_{ii}$, and -- by Lemma 2.1(a) -- \; $ m_{ii}, i \in \ell_\infty,$ are the lines through $\infty$, it follows that the blocks of $D_R^A$ are precisely the complements of the lines of $\Pi$ through the point $\infty$. Also, clearly, the blocks of the column design $D_C^A$ of $A$ are the lines $c_j(t), (t,j) \in (\ell_\infty \setminus \{ \infty \}) \times (\ell_\infty \setminus \{ \infty \}).$  Since any line not passing through $\infty$ meets the line $\ell_\infty$ at some point $t \neq \infty$, it follows that the blocks of $D_C^A$ are precisely the lines not passing through $\infty$. Thus, $(D_R^A, D_C^A) = \mathbb{A}(\Pi, \infty)$ (recall Notation 1.1). Therefore, the extremal un-ordered triple array  $\mathbb{A}(\Pi, \infty)$ is orderable and,  by Theorem 1.2, $A$ is an $(n+1) \times n^2$ extremal triple array. $\Box$

\textbf{ Example: the desarguesian case.}

We now proceed to present an explicit description of the triple array constructed above in the case where $\Pi$ is a desarguesian plane, i.e., projective plane over a finite field. To set up our notations, we begin by recalling the usual construction of these planes by the projectivization of the corresponding affine planes.

\begin{nota} Let $s$ be a prime power, and $F_s$ be the field of order s. Let $AG(2,s)$ denote, as usual, the affine plane of order $s$, described by cartesian co-ordinates $(X,Y)$ over this field. Thus the point set of $AG(2,s)$ is $F_s \times F_s$, and the lines of $AG(2,s)$ are given by inhomogeneous linear equations (with coefficients from $F_s$) in the variables $X,Y.$ We put $\ell_\infty = \{ \infty \} \cup F_s$. Also, for any line $\ell $ of $AG(2,s)$, we define the \textbf{ projective closure} $ \overline{\ell}$ of $\ell$ as the set $\{t\} \cup \ell$, where $t \in \ell_\infty$ is the slope of $\ell$. The points of the projective plane $PG(2,s)$ over $F_s$ are the points of $AG(2,s)$ as well as the elements of $\ell_\infty$ (the points at infinity). The lines of $PG(2,s)$ are $\ell_\infty$ (the line at infinity) as well as the projective closures  of the lines of $AG(2,s)$.
\end{nota}

From now on, $s$ is an odd prime power. We choose the oval $C$ to be the conic in $PG(2,s)$ given as $C := \{ \infty \} \cup \{ (x,y) \in F_s \times F_s: x^2=4(y-x) \}. $ In other words, in terms of homogeneous co-ordinates $[x:y:z]$, $C$ is the zero-set, viewed projectively, of the non-degenerate quadratic form $(x,y,z) \mapsto x^2 + 4(x-y)z$. Clearly, the tangent line to $C$ at the point $\infty$ is the line $\ell_\infty$ of Notation 2.4. The choice of $\infty$ as the base point and the choice of the oval $C$ are  immaterial since by Segre's celebrated theorem in \cite{seg}, all the ovals in $PG(2,s)$, $s$ odd, are conics and since the collineation group $PGL(3,s)$ of $PG(2,s)$ is transitive
on the pairs $(\infty, C)$ where $\infty$ is a point and $C$ is a conic through $\infty$. Our construction also involves an indexing of the tangents to the oval $C$ by points of $\ell_\infty$ (where the tangent at $\infty$ is labelled as $\ell_\infty$), and it may appear at first glance that, given the pair $(C, \infty)$, different choices of this indexing may yield non-isomorphic arrays. That is not the case. Any two of these indexing are related by a permutation $\tau$ of $\ell_\infty = F_s \cup \{ \infty \}$ such that $\tau (\infty)= \infty$. Also, as we shall see, the rows of the arrays constructed are indexed by $\ell_\infty$, and their columns are indexed by $F_s \times F_s$. If $A$ and $B$ are the arrays corresponding to the two choices of the indexing then $B$ is obtainable from $A$ by the simultaneous application of the permutation $\tau$ to the rows of $A$ and of the permutation $(t,j) \mapsto (\tau(t), j)$ to the columns of $A$. Thus different choices yield isomorphic triple arrays. (We recall that two arrays are said to be isomorphic if one can be obtained from the other by a combination of renaming of the symbols and permutations of rows and columns.)

Define  the lines $\ell_t, t \in F_s \cup \{ \infty \},$ as follows. The line $\ell_\infty$ is as in Notation 2.4 and, for $t \in F_s$, $\ell_t$ is the projective closure of the affine line given by the equation $Y= (t+1)X-t^2$. From the choice of the oval $C$ it is clear that $\ell_\infty$ is the tangent to $C$ at the point $\infty$. To show that, for any $t \in F_s$, $\ell_t$ is a tangent to $C$, it suffices to observe that the pair of simultaneous equations $x^2 = 4(y-x)$ and $y=(t+1)x-t^2$ has a unique solution, namely $(x,y)=(2t, t(t+2))$. Since $C$ has exactly $n+1$ tangents, it follows that $\ell_t, t \in \ell_\infty,$ are precisely the tangent lines to $C$. Notice that we indeed have $t \notin \ell_t$ for $t \neq \infty$. In this instance, the derangement $\sigma$ of Notation 2.1 is given by $\sigma (t)=t+1, \, t \in F_s$.

 For  $t \neq \infty$, the affine points on $\ell_t$ are the points $(i+t, it +i+t),\, i \in F_s$. Thus, for $i \neq t \in F_s$, the common point of $\ell_i$ and $\ell_t$ is the point $x_{it}=(i+t, i+t + it )$. Also, for $t \neq \infty$, the point of contact between $\ell_t$ and $C$ is the point
$x_{tt}=(2t, t(t+2))$. Therefore, in this case, the lines $m_{it}$ of Notation 2.1 are given as follows. $m_{i\infty} = \ell_\infty = m_{\infty i}$ for all $i$, and, for $i,t \neq \infty$, $m_{it}$ is the projective closure of the affine line given by the equation $X=i+t$.

Let us parametrize the lines other than $\ell_\infty$ passing through the point $t \in F_s$ as $c_j(t), j \in F_s,$ where $c_j(t)$ is the projective closure of the affine line given by the equation $Y= tX +j.$ Notice that $x_{it}=(i+t, it+i+t) \in c_j(t)$  if and only if $it+i+t = t(i+t)+j$, i.e., if and only if $j=i+t(1-t)$. Therefore, following Notation 2,3, we  have that, for $t \neq \infty$, the permutation $\sigma_t$ of $\ell_\infty$ is now given by the formula ( $\sigma_t(\infty)= \infty$ and)  for $i \neq \infty$,

 \begin{align*} \sigma_t(i)= i + t(1-t). \end{align*}   The rows and columns of the $(s+1) \times s^2$ array $A$ of Theorem 2.1 are now indexed by the elements of  $F_s \cup \{ \infty \}$ and of $F_s \times F_s$, respectively. Following Equation \ref{equation}, let $\ell$ be the line of $PG(2,s)$ given as follows. (i) If $ i=\infty$, $\ell=m_{tt}$, the projective closure of the affine line $X=2t$. (ii) For $i=t$, $\ell=\ell_t$, the projective closure of the affine line $Y=(t+1)X-t^2$. (iii) For $i \neq t, \infty$, $\ell=m_{it}$, the projective closure of the affine line $X=i+t$. Then the common point $z_{ij}(t)$ between $\ell$ and the projective closure $c_j(t)$ of the line $Y=tX+j$ is given by: (i) $z_{ij}(t)= (2t,2t^2+j)$ when $i=\infty$, (ii) $z_{ij}(t) = (t^2+j, t(t^2+j)+j)$ when $i=t$ and (iii) $z_{ij}(t)=(t+i, t(t+i)+j)$ when $i \neq t, \infty$. Hence Equation \ref{equation} yields that, in this case, the triple array $A$ is given explicitly as follows. For $i \in F_s \cup \{ \infty \}, (t,j) \in F_s \times F_s$, the $(i, (t,j))$-th entry of $A$ is given by

$$ A(i, (t,j)) = A_t(i,j)= \begin{cases} (2t, 2t^2 +j) & \text{if} \; i = \infty,\\
t & \text{if} \; j-i = t(1-t), \\
(t^2+j, t(t^2+j)+j) & \text{if} \; i=t,j -i \neq t(1-t), \\
(t+i, t(t+i)+j) & \text{otherwise}.
\end{cases} $$

By Theorem 2.1, this is an $(s+1) \times s^2$ extremal triple array for any odd prime power $s$. Its pair of marginals is the un-ordered triple array $\mathbb{A}(PG(2,s), \infty)$.

\pagebreak

\textbf{A concrete instance}. In particular, taking $s=5$, we get the $6 \times 25$ extremal triple array $A$ given by:

A= \( \begin{pmatrix} A_0 & A_1 & A_2 & A_3 & A_4 \end{pmatrix} \)

where
$$
 A_0=  \begin{pmatrix} 0 & (1,1) & (2,2) & (3,3) & (4,4) \\
                     (1,0) & 0 & (1,2) & (1,3) & (1,4) \\
                     (2,0) & (2,1) & 0 & (2,3) & (2,4) \\
                     (3,0) & (3,1) & (3,2) & 0 & (3,4) \\
                     (4,0) & (4,1) & (4,2) & (4,3) & 0 \\
                     (0,0) & (0,1) & (0,2) & (0,3) & (0,4)
                     \end{pmatrix},
 $$                    $$
 A_1 = \begin{pmatrix} 1 & (1,2) & (1,3) & (1,4) & (1,0) \\
                       (1,1) & 1 & (3,0) & (4,2) & (0,4) \\
                       (3,3) & (3,4) & 1 & (3,1) & (3,2) \\
                       (4,4) & (4,0) & (4,1) & 1 & (4,3) \\
                       (0,0) & (0,1) & (0,2) & (0,3) & 1 \\
                       (2,2) & (2,3) & (2,4) & (2,0) & (2,1)
          \end{pmatrix},
                     $$

$$
A_2 = \begin{pmatrix} (2,4) & (2,0) & (2,1) & 2 & (2,3) \\
                    (3,1) & (3,2) & (3,3) & (3,4) & 2 \\
                    2 & (0,1) & (1,4) & (2,2) & (3,0) \\
                    (0,0) & 2 & (0,2) & (0,3) & (0,4) \\
                    (1,2) & (1,3) & 2 & (1,0) & (1,1) \\
                    (4, 3) & (4,4) & (4,0) & (4,1) & (4,2)
         \end{pmatrix},
         $$
         $$
        A_3  = \begin{pmatrix} (3,4) & (3,0) & (3,1) & (3,2) & 3 \\
                               3 &  (4,3) & (4,4) & (4,0) & (4,1) \\
                               (0,0) & 3  & (0,2) & (0,3) & (0,4) \\
                               (4, 2) & (0, 1) & 3 & (2, 4) & (3, 3) \\
                               (2,1) & (2,2) & (2,3) & 3 & (2,0) \\
                               (1,3) & (1,4) & (1,0) & (1,1) & (1,2)
                               \end{pmatrix},
        $$

$$
A_4= \begin{pmatrix}
     (4,1) & (4,2) & (4,3) & 4 & (4,0) \\
     (0,0) & (0,1) & (0,2) & (0,3) & 4 \\
     4   & ((1,0) & (1,1) & (1,2) & (1,3) \\
     (2,3) & 4    & (2,0) & (2,1) & (2,2) \\
     (1,4) & (2,4) & 4    & (4, 4) & (0,4) \\
     (3,2) & (3,3) & (3,4) & (3,0) & (3,1)
     \end{pmatrix}.
     $$

\end{document}